\documentclass[twoside,11pt]{amsart}

\usepackage{a4wide}
\usepackage{latexsym}
\usepackage{mathrsfs}
\usepackage[T1]{fontenc}
\usepackage{enumitem}
\usepackage{undertilde}

\usepackage{amssymb}
\usepackage{latexsym}
\usepackage{amsmath}
\usepackage{fancyhdr}
\usepackage{bbm}
\usepackage{undertilde}

\usepackage{stackengine}
\usepackage{graphicx}
\usepackage{scalerel}

\numberwithin{equation}{section}

\newtheorem{theorem}{Theorem}[section] 
\newtheorem{lemma}[theorem]{Lemma}     
\newtheorem{corollary}[theorem]{Corollary}
\newtheorem{proposition}[theorem]{Proposition}

\theoremstyle{definition}

\newtheorem{definition}[theorem]{Definition}

\newcommand {\C}{\mathbb C}

\newcommand {\N}{\mathbb N}
\newcommand {\Z}{\mathbb Z}

\newcommand{\rad}{{\rm rad\,}}
\newcommand{\supp}{{\rm supp\,}}

\newcommand{\spn}{{\rm span \,}}

\newcommand{\im}{\operatorname{im}}

\newcommand{\approxlhd}{\mathrel{\raisebox{0.7 pt}{\topinset{$\hstretch{0.9}{\vstretch{0.5}{\sim}}$}{$\lhd$}{6.8 pt}{}}}}
\newcommand{\rank}{\operatorname{rank}}

\newcommand{\Asn}{\mathcal{A}_{sn}}
\newcommand{\Jsn}{\mathcal{J}_{sn}}
\newcommand{\D}{\mathcal{D}}

\title[algebras associated with invariant means]{Algebras associated with invariant means on the subnormal subgroups of an amenable group}

\author[J.\ T.\  White]{Jared T.\ White}
\address{
	Jared T. White, School of Mathematics and Statistics, The Open University, Walton Hall,
	Milton Keynes, MK7 6AA, United Kingdom.}
\email{jw65537@gmail.com; jared.white@open.ac.uk.}

\keywords{Amenable group, invariant mean, just infinite group, branch group, Banach algebra, Arens product, Jacobson radical}

\date{2020}

\subjclass[2010]{43A07 (primary); 43A20,  20E08 (secondary)}

\begin{document}

\begin{abstract}
	Let $G$ be an amenable group. We define and study an algebra $\Asn(G)$, which is related to invariant means on the subnormal subgroups of $G$. For a just infinite amenable group $G$, we show that $\Asn(G)$ is nilpotent if and only if $G$  is not a branch group, and in the case that it is nilpotent we determine the index of nilpotence. We next study $\rad \ell^1(G)^{**}$ for an amenable branch group $G$, and show that it always contains nilpotent left ideals of arbitrarily large index, as well as non-nilpotent elements. This provides infinitely many finitely-generated counterexamples to a question of Dales and Lau \cite{DL}, first resolved by the author in \cite{W2}, which asks whether we always have $(\rad \ell^1(G)^{**})^{\Box 2} = \{ 0 \}$. We further study this question by showing that $(\rad \ell^1(G)^{**})^{\Box 2} = \{ 0 \}$ imposes certain structural constraints on the group $G$.
\end{abstract}

\maketitle

\noindent
\section{Introduction}
\noindent
Let $G$ be an amenable group.
In this article we shall introduce and study a certain algebra $\Asn(G)$, which is defined in terms of objects closely related to invariant means on the subnormal subgroups of $G$, and investigate how algebraic properties of this algebra reflect the subnormal structure of $G$. We shall also study the Jacobson radical of $\ell^1(G)^{**}$. What unites our investigations of these two objects is the importance of nilpotent elements, which we shall construct in essentially the same way in both cases.

The multiplication for our algebra $\Asn(G)$ will be given by the first Arens product coming from $\ell^1(G)^{**}$. This product is a well-known construction in functional analysis and semigroup theory, which in our case can be thought of as extending convolution on $\ell^1(G)$ to the space of finitely additive measures. Since invariant means on subgroups of $G$ may be regarded as elements of $\ell^1(G)^{**}$, one could consider the subalgebra of $\ell^1(G)^{**}$ generated by invariant means on the subnormal subgroups. However, this is not the path we take. Instead we define another object which is similar to an invariant mean but which is better suited to algebra. Specifically, we define an \textit{invariant difference} on a group $G$ to be a left-translation-invariant bounded linear functional $\Gamma$ on $\ell^\infty(G)$ which satisfies $\langle \Gamma, 1 \rangle = 0$.  As we shall point out below, it is known that an infinite group is amenable if and only if it has a non-zero invariant difference. Invariant differences are better suited to our purposes than invariant means because clearly linear combinations of invariant differences are invariant differences, a property that of course fails for invariant means. In fact more is true: given an invariant difference $\Gamma \in \ell^1(G)^{**}$ the left ideal generated by $\Gamma$ consists of invariant differences. We define our algebra $\Asn(G)$ to be the subalgebra of $\ell^1(G)^{**}$ generated by the invariant differences on the subnormal subgroups of $G$. Invariant differences have been studied before in a variety of contexts, but the terminology is ours.

Our initial motivation for studying $\Asn(G)$ was to study $\rad \ell^1(G)^{**}$, and we subsequently discovered the striking way in which $\Asn(G)$ reflects the subnormal structure of $G$, which we believe makes it of independent interest. Moreover, Duchesne, Tucker-Drob, and Wesolek \cite{DTW} have recently used the first Arens product (which they call convolution) with conjugation-invariant means in a proof which characterizes inner amenability of HNN-extensions and graph products of groups. This suggests to us that a better understanding of the algebraic properties of objects related to invariant means could have applications to group theory further down the line. 

The first of our three main results concerns a just infinite amenable group $G$. For this class of groups we are able to make use of the structural results of Wilson \cite{Wi71} and Grigorchuk \cite{Gr00} to study $\Asn(G)$. The definitions of a branch group and of the structure lattice $\mathcal{L}(G)$ will be given in Section 2.

\begin{theorem}		\label{00.1}
	Let $G$ be an amenable just infinite group. Then $\Asn(G)$ is nilpotent if and only if $G$ is not a branch group. Moreover
	\begin{enumerate}
		\item[\rm (i)] if $G$ is finitely-generated and virtually abelian then the index of nilpotence of $\Asn(G)$ is given by $\rank G + 1$;
		\item[\rm (ii)] otherwise the index of nilpotence of $\Asn(G)$ is given by $\log_2 \vert \mathcal{L}(G) \vert + 1$ (which is infinite in the case of a branch group).
	\end{enumerate}
\end{theorem}

 We shall also prove stronger versions of Theorem \ref{00.1} with $\Asn(G)$ replaced by $\Jsn(G)$, the left ideal that it generates in $\ell^1(G)^{**}$, in the case that $G$ is (not necessarily just infinite and) finitely-generated and virtually abelian (Theorem \ref{3.5b}), or hereditarily just infinite (Theorem \ref{2.3}).

 The connection between $\Asn(G)$ and $\rad \ell^1(G)^{**}$ is in part historical. The only known method of constructing elements of $\rad \ell^1(G)^{**}$ is by manipulating invariant differences on subgroups. As such, many published theorems about $\rad \ell^1(G)^{**}$ are in a sense theorems about $\Asn(G)$ or $\Jsn(G)$.
 
 Later on in our article we shall study $\rad \ell^1(G)^{**}$ for an amenable branch group $G$. Our main result is the following.
 
 \begin{theorem}		\label{00.2}
 	Let $G$ be an amenable branch group and let $A$ be either $\ell^1(G)^{**}$ or $\Jsn(G)$. Then $A$ contains nilpotent left ideals of arbitrarily large index. Hence $\ell^1(G)^{**}$ contains non-nilpotent radical elements.
 \end{theorem}

Let us give the background to this result. In \cite[Chapter 14, Question 3]{DL} Dales and Lau asked whether we always have  $(\rad\ell^1(G)^{**})^{\Box 2} = \{ 0 \}$ for every group $G$. This was answered in the negative by the author in \cite{W2}, where it was observed that $\Z^2$ is a counterexample. It was also shown in that paper that the group $G = \oplus_{\N} \Z$ has the stronger property that $\rad(\ell^1(G)^{**})$ is not even nilpotent. It seemed natural to ask whether there are finitely-generated groups with this property. Theorem \ref{00.2} confirms that there are, since many finitely-generated amenable branch groups are known, including the Grigorchuk group and the Gupta-Sidki $p$-groups. 

It remains an open question whether or not there exist any infinite groups $G$ satisfying $(\rad\ell^1(G)^{**})^{\Box 2} = \{ 0 \}$, and even whether this holds when $G=\Z$. Our final theorem puts some restrictions on the possible groups $G$ that could have this property.

\begin{theorem}		\label{00.5}
	Let $G$ be an amenable, just infinite group. If $(\rad\ell^1(G)^{**})^{\Box 2} = \{ 0 \}$ then $G$ is hereditarily just infinite. 
\end{theorem}

The paper is organised as follows. In Section 2 we set out our notation, and recall some background. In particular, Section 2.2 recalls the basic theory of just infinite groups and branch groups, whereas Section 2.3 gives the relevant background concerning Banach algebras and Arens products. In Section 3 we define the algebras $\Asn(G)$ and $\Jsn(G)$ and prove some results about invariant differences for a general group $G$. Section 4 is devoted to understanding when $\Asn(G)$ and $\Jsn(G)$ are nilpotent, in the case that $G$ is an amenable just infinite group or a virtually abelian group. Theorem \ref{00.1} is proved at the end of this section. Finally, in Section 5 we prove Theorems \ref{00.2} and \ref{00.5}, as well as Proposition \ref{5.2}, which states that when $G$ is an amenable branch group, $\rad \ell^1(G)^{**}$ contains nilpotent elements (not just left ideals) of arbitrarily large index.

\section{Background and Notation}
\subsection{General Notation}
$\empty$
For us $\N = \{1,2,3,\ldots\}$, and we denote the group of integers by $\Z$. Given a set $E$ and a subset $F$ we write $\chi_F \colon E \rightarrow \{ 0,1 \}$ for the characteristic function of $F$. 

Henceforth all algebras will be associative and taken over the complex numbers, although they need not be unital. Let $A$ be an algebra. We say that $A$ is \textit{nilpotent} if $A^n = \{ 0 \}$ for some $n \in \N$, and we call the smallest such $n$ the \textit{index of nilpotence}. Now suppose that $A$ is unital. We define the \textit{Jacobson Radical} of $A$, denoted by $\rad A$, to be the intersection of the maximal left ideals of $A$; many alternative characterisations are available (see e.g. \cite[Section 1.5]{D}). The Jacobson radical is a two-sided ideal of $A$. The main fact about the Jacobson radical that we shall require in this article is that every nilpotent left ideal of $A$ is contained in $\rad A$.

\subsection{Group Theory}
$\empty$
We denote 
the symmetric group on $k$ elements by $S_k$ and 
the infinite dihedral group by $D_\infty$. We write the trivial group as $1$ and typically we denote the identity element of a group by $e$. Let $G$ be a group, and let $t, u \in G$. We write conjugation of $u$ by $t$ as $u^t := tut^{-1}$.

We recall that a subgroup $H$ of $G$ is said to be \textit{subnormal}, written here as $H \approxlhd G$, if there exists $n \in \N$, and subgroups $H_1, \ldots, H_n \leq G,$ such that
$$H = H_n \lhd H_{n-1} \lhd \dots \lhd H_1 \lhd G.$$

We write the index of a subgroup $H$ of a group $G$ as $[G:H]$. An infinite group for which every non-trivial normal subgroup has finite index is said to be \textit{just infinite}. The group $G$ is \textit{hereditarily just infinite} if every non-trivial normal subgroup of $G$ is just infinite; by \cite[Proposition 4]{Wi71} this is equivalent to asking that every subnormal subgroup of $G$ has finite index. Examples of hereditarily just infinite groups include $\Z$ and $D_\infty$.

We record the following easy Lemma, that is probably well-known to experts.

\begin{lemma}		\label{5.4}
	Let $G$ be a just infinite group with a finite index, hereditarily just infinite subgroup. Then $G$ is hereditarily just infinite.
\end{lemma}

\begin{proof}
	Let $H$ be a finite index, hereditarily just infinite subgroup of $G$, which without loss of generality is normal,  and let $1 \neq K_2 \lhd K_1 \lhd G$.
	It is sufficient to prove that $[G : K_2] < \infty$ (which in fact is also necessary since $G$ is just infinite). Indeed, $K_2 \cap H$ is either trivial or finite index in $H$. If  $K_2 \cap H = 1$ then $K_2 \cong HK_2/H$, which is finite, contradicting \cite[Proposition 3]{Wi71}. Hence $[H : K_2 \cap H]<\infty$ which implies that $K_2$ is finite index in $G$. 
\end{proof}

An important class of groups for us will be branch groups, which we define now. Note that there are other variations of this definition that do not quite agree with ours, which is taken from \cite[Definition 1.1]{BGS}.

\begin{definition}		\label{00.3}
	A group $G$ is said to be a \textit{branch group} if there exist descending chains of subgroups $H_1 \supset H_2 \supset \cdots$ and $L_1 \supset L_2 \supset \cdots$, and a sequence of natural numbers $(k_i)$, such that 
	\begin{enumerate}
		\item[\rm (1)] we have $\bigcap_{i=1}^\infty H_i = 1$;
		\item[\rm (2)] each $H_i$ is a finite index normal subgroup of $G$;
		\item[\rm (3)] each $H_i$ is equal to a direct product
		\begin{equation} 	\label{eq00.3}
			L_i^{(1)} \times \cdots \times L_i^{(k_i)},
		\end{equation}
		 where each of the factors is isomorphic to $L_i$;
		\item[\rm (4)] for each $i \in \N$ the number $m_i := k_{i+1}/k_i$ is an integer which is at least 2, and the direct product decomposition \eqref{eq00.3} for $H_{i+1}$ properly refines that of $H_i$, in the sense that each factor $L_i^{(j)}$ contains the subgroups $L_{i+1}^{(l)}$ for $l = (j-1)m_{i+1}+1, \ldots, jm_{i+1}$;
		\item[\rm (5)] the action of $G$ on $H_i$ by conjugation transitively permutes the subgroups $L_i^{(1)}, \ldots, L_i^{(k_i)}$.
		\end{enumerate}
\end{definition}

We refer to the sequence of subgroups and natural numbers $(H_i, L_i, k_i)_{i \in \N}$ as a \textit{branch structure} of $G$.

Grigorchuk \cite{Gr00}, building on the work of Wilson \cite{Wi71}, showed that every just infinite group is either a just infinite branch group, or else has a finite index subgroup which is isomorphic to a direct product of a finite number of copies of a group $H$, which is either simple or residually finite and hereditarily just infinite. We prefer to use a version of this result that is closer to Wilson's original statement as follows. Recall that the \textit{Baer radical} of a group $G$ is the subgroup generated by the subnormal cyclic subgroups of $G$.

\begin{theorem}		\label{00.4}
	Let $G$ be a just infinite group. Then $G$ satisfies one of the following:
	\begin{enumerate}
		\item[\rm (1)] $G$ is a just infinite branch group;
		\item[\rm (2)] $G$ is a just infinite, finitely-generated virtually abelian group;
		\item[\rm (3)] $G$ has a finite index normal subgroup $H$, which may be expressed as direct product $H = H_1 \times \cdots \times H_n$ of mutually isomorphic groups $H_i \ ( i = 1, \ldots, n)$, which are hereditarily just infinite with trivial Baer radical, and the action of $G$ on $H$ by conjugation transitively permutes these factors.
	\end{enumerate}
\end{theorem}

For a just infinite group with trivial Baer radical, which corresponds to cases (1) and (3) above, Wilson \cite{Wi71} defines the \textit{structure lattice} of $G$ to be the set of subnormal subgroups of $G$ modulo the equivalence relation of commensurability; that is $H,K \approxlhd G$ are equivalent if $[H: H \cap K ]< \infty$ and $[K: H \cap K ]< \infty$. We denote the structure lattice of $G$ by $\mathcal{L}(G)$. It is a lattice with respect to the operations
$$[H] \wedge [K] = [H\cap K] \qquad \text{ and } \qquad [H] \vee [K] = [ \langle H, K \rangle ],$$
where $[H]$ and $[K]$ denote the equivalence classes of $H,K \approxlhd G$ respectively. Wilson completely classifies the isomorphism type of $\mathcal{L}(G)$, and it follows from his work that $\mathcal{L}(G)$ is infinite if and only if $G$ is a branch group.

The trichotomy of Theorem \ref{00.4} will be an important tool for our study of $\Asn(G)$, for a just infinite group $G$. Since we focus on amenable groups in this article, let us point out that each of the three classes of groups appearing in the trichotomy contains interesting examples which are amenable. Indeed, the most famous examples of branch groups, namely the Grigorchuk group and the Gupta-Sidki p-groups, are amenable \cite{BKN} and just infinite \cite[Theorem 4]{Gr00}. For the second class we note that any wreath product of the form $D_\infty \wr F$, where $F$ is a finite group, is amenable and just infinite \cite[Construction 1.(a)]{Wi71}, and contains a finite index copy of $\Z^{\vert F \vert}$. For the final class we can take any group of the form $S \wr F$, where $F$ is again a finite group, and $S$ is an infinite amenable simple group. Note that the first two examples are also finitely-generated, and the third example is finitely-generated if $S$ is.  Infinite finitely-generated amenable simple groups have been constructed by Juschenko and Monod \cite{JM}.

\subsection{Arens Products and Banach Algebras}
$\empty$
We shall denote the dual space of a Banach space $E$ by $E^*$, and it's bidual by $E^{**}$. Given $x \in E$ and $\lambda \in E^*$ we write $\langle x, \lambda \rangle$ for the value of $\lambda$ applied to $x$; however for $\lambda \in E^*$ and $\Phi \in E^{**}$ we typically write $\langle \Phi, \lambda \rangle$ for the value of $\Phi$ applied to $\lambda$ in order to more easily identify $E$  with its canonical image inside $E^{**}$.

Given a group $G$ we define 
$$\ell^1(G) = \left\{ f \colon G \rightarrow \C : \sum_{t \in G} \vert f(t) \vert < \infty \right\}.$$
This is a Banach space with the norm of $f \in \ell^1(G)$ given by $\Vert f \Vert = \sum_{t \in G} \vert f(t) \vert$, and becomes a Banach algebra with multiplication given by convolution: 
$$(f*g)(t) = \sum_{s \in G} f(s)g(s^{-1}t) \qquad (t \in G).$$

 Given another group $H$, and a group homomorphism $\varphi \colon G \rightarrow H$, we shall often identify $\varphi$ with the bounded algebra homomorphism $\varphi \colon \ell^1(G) \rightarrow \ell^1(H)$ that it induces; however, we shall write 
$\varphi^{**} \colon \ell^1(G)^{**} \rightarrow \ell^1(H)^{**}$ for the second conjugate map.

Let $A$ be a Banach algebra. We recall that the first Arens products $\Box$ on $A^{**}$ is defined in three stages via the following formulae: 
$$
\langle \Psi \Box \Phi, \lambda \rangle = \langle \Psi, \Phi \cdot \lambda \rangle; \qquad
\langle a, \Phi \cdot \lambda \rangle = \langle \Phi, \lambda \cdot a \rangle; \qquad
\langle b, \lambda \cdot a \rangle  = \langle ab, \lambda \rangle, 
$$
for $\Phi, \Psi \in A^{**}, \lambda \in A^*,$ and $a, b \in A$. The product $\Box$ makes $A^{**}$ into a Banach algebra and agrees with the original multiplication on $A$ when it is identified with its image in $A^{**}$ under the canonical embedding. Right multiplication with a fixed element with respect to $\Box$ is weak*-continuous, but left multiplication can fail to be weak*-continuous in general.
Applying Goldstein's Theorem, this leads to the following formula for the first Arens product:
$$\Phi \Box \Psi = \lim_\alpha \lim_\beta a_\alpha b_\beta,$$
where $(a_\alpha)$ and $(b_\beta)$ are nets in $A$ converging to $\Phi$ and $\Psi$ respectively, and all limits are taken with respect to the weak*-topology.

There is also a second Arens product, $\Diamond$, which is again extends the given multiplication on $A$, but is weak*-continuous on the left instead of the right. The two products can be different, and always are when $A = \ell^1(G)$, for an infinite group $G$. 
In this article we make the arbitrary choice to study $\Box$ instead of $\Diamond$; analogous results to those we prove here should hold for $\Diamond$, with left and right swapped in appropriate places. For more background on Arens products see \cite{DL,Pa94}.

We extend our notation for conjugation by group elements from the group itself to $\ell^1(G)^{**}$ and write $\Phi^t = \delta_{t}\Box \Phi \Box \delta_{t^{-1}} \ (t \in G, \ \Phi \in \ell^1(G)^{**})$.


Let $H$ be a subgroup of $G$. We say that $\Phi \in \ell^1(G)^{**}$ is \textit{$H$-left-invariant} if $\delta_t \Box \Phi = \Phi \ (t \in H)$. If $H=G$ we might simply say that $\Phi$ is left-invariant. Also we shall always use the phrase ``invariant mean'' to mean a left-invariant mean, since we never consider right-invariant means in this paper; we take the same approach when we define invariant differences in Section 3 below. 

Write $\iota \colon \ell^1(H) \hookrightarrow \ell^1(G)$ for the inclusion map. In what follows we shall say that $\Phi \in \ell^1(G)^{**}$ is \textit{supported on $H$} if $\Phi \in \im \iota^{**}$. Equivalently, $\Phi$ is supported on $H$ as a finitely additive measure. When there is no ambiguity we usually identify $\ell^1(H)$ with its image inside $\ell^1(G)$, and $\ell^1(H)^{**}$ with its image inside $\ell^1(G)^{**}$, which coincides with 
$\overline{\im \iota}^{\, w^*}$. Hence, with these identifications, we may write $\ell^1(H)^{**} = \overline{\ell^1(H)}^{\, w^*}$. Similarly, for $t \in G$, we may identify $\ell^1(Ht)^{**}$ with a weak*-closed linear subspace of $\ell^1(G)^{**}$, which in fact coincides with $\ell^1(H)^{**} \Box \delta_t$.

We record the following lemma, which tells us how to break up an element of $\ell^1(G)^{**}$ in terms of cosets of a subgroup. The proof is routine.

\begin{lemma} 	\label{0.1}
	Let $G$ be a group with a finite index subgroup $H$. Let $n = [G:H]$ and let 
	$\{ t_1, t_2, \ldots, t_n \}$ be a right transversal for $H$ in $G$.  Then
	\begin{enumerate}
		\item[\rm (i)] $\ell^1(G)^{**} = \bigoplus_{j=1}^n \ell^1(Ht_j)^{**}$ as Banach spaces;
		\item[\rm (ii)] $\ell^1(Ht_j)^{**} = \ell^1(H)^{**} \Box \delta_{t_j} \ (j=1, \ldots, n).$
		In particular, each $\Phi \in \ell^1(G)^{**}$ can be written uniquely as 
		\begin{equation}		\label{eq1.1}
		\Phi = \Phi_1 \Box \delta_{t_1} + \cdots + \Phi \Box \delta_{t_n},
		\end{equation}
		for some $\Phi_1, \ldots, \Phi_n \in \ell^1(H)^{**}$.
	\end{enumerate}
	The analogous results hold for left transversals.
\end{lemma}



We refer to equation \eqref{eq1.1} as the \textit{right coset decomposition of $\Phi$ with respect to $t_1, \ldots, t_n$},  and we call the elements $\Phi_1, \ldots, \Phi_n$ appearing in the decomposition as the \textit{right coset factors of $\Phi$}. We refer similarly to the left coset decomposition with respect to a left transversal.

\section{Invariant Differences and the Algebra $\Asn(G)$}
\noindent
In this section we shall define invariant differences and the algebra $\Asn(G)$, and prove some lemmas about their properties for a general group $G$.

For a group $G$, we define an \textit{invariant difference} on $G$ to be an element 
$\Gamma \in \ell^1(G)^{**}$ such that $\delta_s \Box \Gamma = \Gamma \ (s \in G)$, and 
$\langle \Gamma, 1 \rangle = 0$. Note that ``invariant'' here is again used to mean ``left-invariant''. We denote the set of invariant differences on $G$ by $\D(G)$. When $G$ is a finite group $\D(G) = \{ 0 \}$.
On the other hand, in the infinite setting $\D(G)$ is an interesting object: an infinite group has a non-zero invariant difference if and only if the group is amenable. More precisely, there is a relationship between invariant means and invariant differences as follows. If we are given an infinite amenable group $G$ it is known that there are infinitely many (indeed $2^{2^{\vert G \vert}})$ invariant means on $G$. Given any two of these invariant means, $M_1$ and $M_2$ say, their difference $M_1 - M_2$ is an invariant difference on $G$, as is any scalar multiple of this. On the other hand, given $\Gamma \in \D(G)$ we may consider $\Gamma$ as a measure on the Stone--{\v C}ech compactification $\beta G$. Its total variation measure $\vert \Gamma \vert$ will then be a $G$-left-invariant positive measure on  $\beta G$, and in the case that $\Gamma \neq 0$ scaling it by $\vert \Gamma \vert (\beta G)^{-1}$ gives an invariant mean on $G$.  

We summarize some well known properties of $\D(G)$. Fairly accessible proofs of more general results can be found in \cite[Proposition 8.23]{DL} and \cite[Proposition 4.2]{W2}. 

\begin{lemma}		\label{2.0}
	The set $\D(G)$ is a weak*-closed ideal of $\ell^1(G)^{**}$, and every $\Gamma \in \D(G)$ satisfies $\Phi \Box \Gamma = \langle \Phi, 1 \rangle \Gamma \ ( \Phi \in \ell^1(G)^{**})$. Hence $\D(G)^{\Box 2} = \{ 0 \}$.
\end{lemma}

We define $\Asn(G)$ to be the subalgebra of $\ell^1(G)^{**}$ generated by invariant differences on subnormal subgroups of $G$, that is
$$\Asn(G) = \spn \{ \Gamma_1 \Box \cdots \Box \Gamma_n : \Gamma_i \in \D(H_i) \ \text{\rm for some } H_i \approxlhd G \ (i = 1,\ldots, n) \}.$$
Similarly, we define $\Jsn(G)$ to be the left ideal of $\ell^1(G)^{**}$ given by
$$\Jsn(G) = \spn \{ \Phi \Box \Gamma : \Phi \in \ell^1(G)^{**}, \ \Gamma \in \D(H) \ \text{\rm for some } H \approxlhd G \}.$$
The algebra $\Asn(G)$ will be our main object of study, but in some cases our results extend to $\Jsn(G)$. 

It is unclear to us whether or not $\Asn(G)$ and $\Jsn(G)$ are closed in either the norm or weak*-topology on $\ell^1(G)^{**}$. However, for the Banach algebraists we note that, since the index of nilpotence of a subalgebra of $\ell^1(G)^{**}$ passes to both its norm- and weak*-closure, all of our theorems about nilpotence of $\Asn(G)$ and $\Jsn(G)$ remain true if they are replaced by their closures in either topology.

We begin by looking at how invariant differences behave under conjugation by a group element.

\begin{lemma}		\label{2.2}
	Let $G$ be a group and let $H \leq G$. For any invariant difference $\Gamma$ on $H$ and any $t \in G$, the functional $\Gamma^t$ is an invariant difference on $H^t$.
\end{lemma}

\begin{proof}
	Note that, since conjugation by $\delta_t$ is a weak*-homeomorphism of $\ell^1(G)^{**}$, we have
	$$(\ell^1(H)^{**})^t = \left(\overline{\ell^1(H)}^{\, w^*} \right)^t 
	= \overline{\ell^1(H)^t}^{\, w^*} = \ell^1(H^t)^{**},$$
	and it follows that $\Gamma^t$ is supported on $H^t$. To see invariance, take $s \in H$, and $\lambda \in \ell^1(H^t)^*$. Then
	$$\langle \delta_{tst^{-1}} \Box \Gamma^t, \lambda \rangle = 
	\langle \delta_t \Box (\delta_s \Box \Gamma)\Box \delta_{t^{-1}}, \lambda \rangle = 
	\langle \delta_s \Box \Gamma, \delta_t \cdot \lambda \cdot \delta_{t^{-1}} \rangle
	= \langle \Gamma, \delta_t \cdot \lambda \cdot \delta_{t^{-1}} \rangle = \langle \Gamma^t, \lambda \rangle.$$
	Finally, $\langle \Gamma^t, \chi_{H^t} \rangle = \langle \Gamma, \chi_H \rangle = 0$.
\end{proof}

We now turn to look at the coset factors of invariant differences. Our next lemma can be summarized as saying that left invariance is preserved by right decompositions, whereas the ``difference property'' $\langle \Gamma, 1 \rangle = 0$ is preserved by left decompositions, and in the case that the subgroup is normal both properties are preserved by either decomposition.

\begin{lemma}		\label{0.1a}
	Let $G$ be a group, and let $H$ be a finite index subgroup of $G$, with $n = [G : H]$. Let $\Lambda \in \ell^1(G)^{**}$.
	\begin{enumerate}
		\item[\rm (i)] Suppose that $\Lambda$ is $G$-left-invariant. If $t_1, \ldots, t_n$ is a right transversal for $H$ in $G$, and 
		$\Lambda = \sum_{i=1}^n \Phi_i \Box \delta_{t_i}$
		is the right coset decomposition of $\Lambda$ with respect to $t_1, \ldots, t_n$, then each $\Phi_i$ is $H$-left-invariant.
		\item[\rm (ii)] Suppose that $\Lambda$ is an invariant difference. If $s_1, \ldots, s_n$ is a left transversal, 
		and 
		$\Lambda = \sum_{i = 1}^n \delta_{s_i} \Box \Psi_i,$
		is the left coset decomposition of $\Lambda$ with respect to $s_1, \ldots, s_n$,
		then for each $i = 1 \ldots, n$ we have $\langle \Psi_i, 1 \rangle = \langle \Psi_i, \chi_H \rangle = 0$.
		\item[\rm (iii)] Suppose now that $H$ is a normal subgroup, and that $\Lambda$ is an invariant difference on $G$. 
		Then each of the coset factors $\Phi_i$ and $\Psi_i$ is an invariant difference on $H$.
	\end{enumerate}
\end{lemma}

\begin{proof}
	(i) Note that for each $u \in H$ we have $\Lambda = \delta_u \Box \Lambda = \sum_{j = 1}^n \delta_u \Box \Phi_i \Box \delta_{t_i}$, which forces $\delta_u \Box \Phi_i = \Phi \ (i=1, \ldots, n)$ by the uniqueness of the decomposition, and the fact that
	$\supp (\delta_u \Box \Phi_i)$ is contained in $H \ (i=1, \ldots, n)$.
	
	(ii) First note that, for $i = 1, \ldots, n$, we have
	$$\langle \Lambda, \chi_H \rangle = \langle \delta_{s_i^{-1}} \Box \Lambda, \chi_H \rangle  = \langle \Psi_i, \chi_H \rangle 
	= \langle \Psi_i, 1 \rangle,$$
	where for the second equality we have used the fact that $H \cap \supp(\delta_{s_i^{-1}s_j} \Box \Psi_i) = \emptyset$, for $j \neq i$. Since $0 = \langle \Lambda, 1 \rangle = \sum_{i = 1}^n \langle \delta_{s_i} \Box \Psi_i, 1 \rangle = \sum_{i=1}^n \langle \Psi_i, 1 \rangle$, and the summands are all equal, it follows that $\langle \Psi_i, 1 \rangle = 0  \ (i = 1, \ldots, n)$.

	(iii) First consider the right coset factors $\Phi_i$. Since $H$ is normal, the right transversal $t_1, \ldots, t_n$ is also a left transversal. Clearly we have
	$\Lambda = \sum_{i=1}^n \delta_{t_i} \Box \Phi_i^{t_i^{-1}}$. By the uniqueness of the left coset decompositon, we may apply (ii) to this formula to see that, for each $i = 1, \ldots, n$, we have $ 0 = \langle \Phi_i^{t_i^{-1}}, 1 \rangle = \langle \Phi_i, t_i \cdot 1 \cdot t_i^{-1} \rangle = \langle \Phi_i, 1 \rangle.$ By (i) we have left invariance, so that each $\Phi_i$ is an invariant difference.
	
	As for the left coset factors, by uniqueness of the decomposition, $\Psi_i^{s_i} \ (i = 1, \ldots, n)$ are the left coset factors with respect to $\{ s_1, \ldots, s_n \}$, so that each $\Psi_i^{s_i}$ is 
	$H$-left invariant by (i). Fix $i \in \{ 1, \ldots, n \}$. For any $t \in H$, we have $\delta_t \Box \Psi_i^{s_i} = \Psi_i^{s_i}$ which implies that $\delta_{s_i^{-1}ts_i} \Box \Psi_i = \Psi_i$. Since $H$ is normal and $t\in H$ was arbitrary, this implies that $\Psi_i$ is $H$-left-invariant. 
	We have $\langle \Psi_i, 1 \rangle = 0$ by (ii).
\end{proof}

\begin{corollary}		\label{0.1b}
	Let $G$ be a group, let $H$ be a finite index subnormal subgroup of $G$, and write $n = [G : H]$. There exists a left transveral 
	$t_1, \ldots, t_n$ for $H$ in $G$ with the property that, for any invariant difference $\Gamma$ on $G$, with right coset decomposition $\Gamma = \sum_{i=1}^n \Phi_i \Box \delta_{t_i}$ with respect to $t_1, \ldots, t_n$, each of the coset factors $\Phi_i \ (i = 1, \ldots, n)$ is an invariant difference on $H$.
\end{corollary}

\begin{proof}
	Since $H$ is subnormal, we can write $H= H_m \lhd H_{m-1} \lhd \cdots \lhd H_1 \lhd H_0 = G$, for some $m \in \N$, and some
	subgroups $H_1, \ldots, H_{m-1}$. For each $i = 1, \ldots, m$, let $r_i = [H_i : H_{i-1}],$ and let $t_{i,1}, \ldots, t_{i, r_i}$ be a right transversal for 
	$H_i$ in $H_{i-1}$. Then for each $i$ the set 
	$$T_i : = \{ t_{i,k_i}t_{i-1, k_{i-1}} \cdots t_{1,k_1} : k_j =1, \ldots, r_j, \ j= 1, \ldots, i \}$$
	is a right transversal for $H_i$ in $G$. It now follows by induction on $i$ that the coset factors with respect to each $T_i$ belong to $\D(H_i)$: for the induction step apply Lemma \ref{0.1a}(iii) to the coset factors with respect to $T_{i-1}$.
	When $i = m$ we obtain the result.
\end{proof}

The following lemmas shall help us to establish properties of products of functionals and invariant differences.

\begin{lemma}		\label{0.4b}
	Let $G$ be a group, and let $H \leq G$, and let $\Phi \in \ell^1(G)^{**}$ be $H$-left-invariant. Then 
	\begin{enumerate}
		\item[\rm (i)] for $\Lambda \in \ell^1(H)^{**}$ we have $\Lambda \Box \Phi = \langle \Lambda, \chi_H \rangle \Phi = \langle \Lambda, 1 \rangle \Phi$;
		\item[\rm (ii)] if $H\lhd G$ then $\Psi \Box \Phi$ is $H$-left-invariant for every $\Psi \in \ell^1(G)^{**}$.
	\end{enumerate}
\end{lemma}

\begin{proof}
	(i) This clearly holds when $\Lambda$ is a point mass, and the general case follows by taking weak*-limits of linear combinations of point masses.
	
	(ii) Let $u \in G$ and $s \in H$. Then
	$$\delta_s \Box \delta_u \Box \Phi = \delta_u \Box \delta_{u^{-1}su} \Box \Phi = \delta_u \Box \Phi.$$
	Hence $\delta_u \Box \Phi$ is $H$-left-invariant, and by taking linear combinations and weak*-limits of point masses we see that $\Psi \Box \Phi$ is $H$-left-invariant for all $\Psi \in \ell^1(G)^{**}$, as required.
\end{proof}

\begin{lemma}		\label{2.1}
Let $G$ be a group, and let $H, K \leq G$, and suppose that $H \cap K $ is finite index in $H$. Let $\Gamma$ be an invariant difference on $H$, and let $\Psi \in \ell^1(G)^{**}$ be $K$-left-invariant. Then $\Gamma \Box \Psi = 0$.
\end{lemma}

\begin{proof}
Let $n= [H : H \cap K] < \infty$, and let $t_1, \ldots, t_n$ be a left transversal for $H \cap K$ in $H$. Let 
$$\Gamma = \delta_{t_1}\Box \Phi_1 + \cdots + \delta_{t_n} \Box \Phi_n$$
be the left coset decomposition of $\Gamma$. Applying Lemma \ref{0.1a}(ii) and Lemma \ref{0.4b}(i), we see that 
$$\Gamma \Box \Psi = \sum_{i=1}^n \delta_{t_i} \Box \Phi_i \Box \Psi
= \sum_{i = 1}^n \langle \Phi_i, 1 \rangle \delta_{t_i} \Box \Psi = 0,$$
as required.
\end{proof}

In particular, the previous lemma implies that, if $\Gamma_H$ and $\Gamma_K$ are invariant differences on subgroups $H$ and $K$ respectively, then $\Gamma_H \Box \Gamma_K = 0$ whenever $[H : H \cap K]$ is finite. The next lemma gives us a method to find invariant differences with non-zero products. Specifically, it implies that non-zero invariant differences on distinct direct factors of a groups have non-zero product. 

\begin{lemma} 	\label{0.4}
	Let $G$ be  a  group, and let $H$ be a subgroup of $G$. Suppose that $H$ may be written as
	$$H = L_1 \times \cdots \times L_k,$$
	for some $k \in \N$, and some subgroups $L_1, \ldots, L_k$ of $G$. For each  $i = 1, \ldots, k,$ let $\Psi_i \in \ell^1(G)^{**} \setminus \{ 0 \}$ be supported on $L_i$. Then 
	$$\Psi_1 \Box \cdots \Box \Psi_k \neq 0.$$
\end{lemma}

\begin{proof}
	For each $i = 1, \ldots, k$ choose $y_i \in \ell^\infty(L_i)$ satisfying $\langle \Psi_i, y_i \rangle \neq 0$. 
	Define $y \in \ell^\infty(G)$ by $y(v) = 0$ if $v \notin H$, and otherwise by
	$$y((u_1, \ldots, u_k)) = y_1(u_1) \cdots y_k(u_k) \quad (u_i \in L_i, \ i=1, \ldots, k).$$
	We clearly  have
	$$\langle \delta_{u_1}* \cdots *\delta_{u_k}, y \rangle = \langle \delta_{u_1}, y_1 \rangle \cdots \langle \delta_{u_k}, y_k \rangle$$
	for $u_i \in L_i \ (i=1, \ldots, k)$, which implies that
	$$\langle f_1* \cdots *f_k, y \rangle = \langle f_1, y_1 \rangle \cdots \langle f_k, y_k \rangle$$
	for $f_i \in \ell^1(L_i) \ (i = 1, \ldots, k)$. After taking weak*-limits, this implies that 
	$$\langle \Phi_1 \Box \cdots \Box \Phi_k, y \rangle = \langle \Phi_1, y_1 \rangle \cdots \langle \Phi_k, y_k \rangle$$
	for $\Phi_i \in \ell^1(L_i)^{**} \  (i = 1, \ldots, k)$. Hence
	$$\langle \Psi_1 \Box \cdots \Box \Psi_k, y \rangle = \langle \Psi_1, y_1 \rangle \cdots \langle \Psi_k, y_k \rangle \neq 0.$$
	In particular $\Psi_1 \Box \cdots \Box \Psi_k \neq 0,$ as required.
\end{proof}

\section{Nilpotence of the algebras $\Asn(G)$ and $\Jsn(G)$}
\noindent
In this section we prove Theorem \ref{00.1}, as well some results about the nilpotence of $\Jsn(G)$.
We have been unable to prove the analogue of Theorem \ref{00.1} for $\Jsn(G)$ in general. However in Theorem \ref{2.3} and Theorem \ref{3.5b} we are able to determine the index of nilpotence of $\Jsn(G)$ for hereditarily just infinite groups and virtually abelian groups respectively.

\begin{theorem}		\label{2.3}
	Let $G$ be an amenable, just infinite group. Then the following are equivalent:
	\begin{enumerate}
		\item[\rm (1)] $G$ is hereditarily just infinite;
		\item[\rm (2)] $\Jsn(G)^{\Box 2} = \{ 0 \}$;
		\item[\rm (3)] $\Asn(G)^{\Box 2} = \{ 0 \}$.
	\end{enumerate}
\end{theorem}

\begin{proof}
	Suppose that $G$ is hereditarily just infinite. By \cite[Proposition 4]{Wi71}, every subnormal subgroup of $G$ has finite index. Let $H, K \approxlhd G$, let $\Gamma_H$ and $\Gamma_K$ be invariant differences on $H$ and $K$ respectively, and let $\Phi, \Psi \in \ell^1(G)^{**}$. We first show that $\Phi \Box \Gamma_H \Box \Psi \Box \Gamma_K = 0$. Let $[G : K] = n$, and pick a left transversal $t_1, \ldots, t_n$ for $K$ in $G$. Write $\Psi = \sum_{i = 1}^n \delta_{t_i} \Box \Psi_i$, for some $\Psi_1, \ldots, \Psi_n$ supported on $K$, and note that $\Psi_i \Box \Gamma_k = \langle \Psi_i, 1 \rangle \Gamma \ (i=1, \ldots, n)$ by Lemma \ref{0.4b}(i). Then, by Lemma \ref{2.1} and Lemma \ref{2.2}, we have
	$$\Gamma_H \Box \Psi \Box \Gamma_K 
	= \Gamma_H \Box \left( \sum_{i = 1}^n \delta_{t_i} \Box \Psi_i \Box \Gamma_K \right)
	=\sum_{i=1}^n \langle \Psi_i, 1 \rangle \delta_{t_i} \Box \Gamma_H^{t_i^{-1}} \Box \Gamma_K = 0,$$
	whence also $\Phi \Box \Gamma_H \Box \Psi \Box \Gamma_K = 0$. Since a general element of $\Jsn(G)$ is a finite sum of elements of the form $\Phi \Box \Gamma_H$, for $H \approxlhd G$, it follows that $\Jsn(G)^{\Box 2} = \{ 0 \}$. We have shown that (1) implies (2). It is trivial that (2) implies (3).
	
	To show that (3) implies (1) we prove the contrapositive. Indeed, Lemma \ref{5.4} and Theorem \ref{00.4} together imply that, if $G$ is not hereditarily just infinite, then $G$ has a subnormal subgroup $H$ such that $H = H_1 \times H_2$, for some infinite $H_1, H_2 \approxlhd G$. Let $\Gamma_i \in \D(H_i) \setminus \{ 0 \} \ (i=1,2)$. Then $\Gamma_1 \Box \Gamma_2 \neq 0$ by Lemma \ref{0.4}.
\end{proof}

Next we prove that the stronger version of Theorem \ref{00.1} with $\Jsn(G)$ in place of $\Asn(G)$ holds when $G$ is any (not necessarily just infinite) virtually abelian group. In what follows we define the \textit{rank} of a finitely-generated virtually abelian group $G$, denoted by $\rank G$, to be the rank of any finite index abelian subgroup.

\begin{lemma}		\label{3.5a}
	Let $A = \Z^n$ for some $n \in \N$, and let $K_1, \ldots, K_{n+1} \leq A$. Let $\Gamma_i \in \D(K_i)$ and $\Phi_i \in \ell^1(A)^{**} \ (i = 1, \ldots, n+1)$.
	Then 
	$$\Phi_1 \Box \Gamma_1 \Box \cdots \Box \Phi_{n+1} \Box \Gamma_{n+1} = 0.$$
\end{lemma}

\begin{proof}
	We shall prove the following stronger statement by induction: let $q = \rank \langle K_1, \ldots, K_{n+1} \rangle$. Then $\Phi_1 \Box \Gamma_1 \Box \cdots \Box \Phi_{q+1} \Box \Gamma_{q+1} = 0$.
For the base case $q=0$ just note that the only invariant difference on the trivial group is $0$.

Suppose that $q>0$. There are two cases. Firstly, if $\rank \langle K_2, \ldots, K_{q+1} \rangle \leq q-1$ then $\Phi_2 \Box \Gamma_2 \Box \cdots \Box \Phi_{q+1} \Box \Gamma_{q+1} = 0$ by the induction hypothesis. Secondly, if $\rank \langle K_2, \ldots, K_{q+1} \rangle = q$, then in particular $\rank \langle K_1, K_2, \ldots, K_{q+1} \rangle = \rank \langle K_2, \ldots, K_{q+1} \rangle$, which implies that 
$$[K_1 : K_1 \cap \langle K_2, \ldots, K_{q+1} \rangle] < \infty.$$
By Lemma \ref{0.4b}(ii) $\Phi_2 \Box \Gamma_2 \Box \cdots \Box \Phi_{q+1} \Box \Gamma_{q+1}$ is $\langle K_2, \ldots, K_{q+1} \rangle$-left-invariant. It follows that 
$$\Gamma_1 \Box \Phi_2 \Box \Gamma_2 \Box \cdots \Box \Phi_{q+1} \Box \Gamma_{q+1} = 0$$
by Lemma \ref{2.1}. This completes the proof.
\end{proof}

\begin{theorem}		\label{3.5b}
	Let $G$ be a finitely-generated, virtually abelian group of rank $n$. Then $\Asn(G)$ and $\Jsn(G)$ are nilpotent of index $n+1$.
\end{theorem}

\begin{proof}
	Let $A$ be a finite index normal subgroup of $G$ isomorphic to $\Z^n$ say, and let $m = [G:A]$. That $\Asn(G)^{\Box n} \neq \{ 0 \}$ follows by considering non-zero invariant differences on each of the direct factors of $A$ and applying Lemma \ref{0.4}. As such $\Jsn(G)^{\Box n} \neq \{ 0 \}$ also. To complete the proof we shall prove that $\Jsn(G)^{\Box (n+1)} = \{ 0 \}$.
	
	Let $H_1, \ldots, H_{n+1} \approxlhd G$, and let $\Gamma_i$ be an invariant difference on $H_i \ (i = 1, \ldots, n+1)$. Let $\Phi_i \in \ell^1(G)^{**} \ (i = 1, \ldots, n)$ be arbitrary. We claim that $\Phi_1 \Box \Gamma_1 \Box \cdots \Box \Phi_{n+1} \Box \Gamma_{n+1} = 0$.
	
	Set $K_i = H_i \cap A \ (i = 1, \ldots, n+1)$, set $m_i = [H_i: K_i]$, and let
	$$\Gamma_i = \sum_{j=1}^{m_i} \Gamma_{i,j} \Box \delta_{t_{i,j}}$$
	be the right coset decomposition of $\Gamma_i$ in terms of $K_i$, for some right transversal $t_{i,1}, \ldots, t_{i,m_i}$ of $K_i$ in $H_i$. By Lemma \ref{0.1a}(iii) each $\Gamma_{i,j}$ is an invariant difference on $K_i$. Also, let $s_1, \ldots, s_m$ be a right transversal for $A$ in $G$, and let 
	$$\Phi_i = \sum_{j=1}^m \Phi_{i,j} \Box \delta_{s_j}$$
	be the right coset decomposition of $\Phi_i \ (i = 1, \ldots, n+1)$. 
	We have 
	\begin{multline*}
	\Phi_1 \Box \Gamma_1 \Box \cdots \Box \Phi_{n+1} \Box \Gamma_{n+1} = \\
	\sum_{j_1} \cdots \sum_{j_{n+1}} \sum_{l_1} \cdots \sum_{l_{n+1}}
	\Phi_{1, j_1} \Box \delta_{s_{j_1}} \Box \Gamma_{1, l_1} \Box \delta_{t_{1,l_1}} \Box \cdots \\ 
	\cdots \Box \Phi_{n+1, j_{n+1}} \Box \delta_{s_{j_{n+1}}} \Box \Gamma_{n+1, l_{n+1}} \Box \delta_{t_{n+1,l_{n+1}}}.
	\end{multline*}
	By passing the point masses to the right, each summand of the above expression can be rewritten in the form
	$$\Psi_1^{u_1} \Box \Lambda_1^{v_1} \Box \cdots \Box \Psi_{n+1}^{u_{n+1}} \Box \Lambda_{n+1}^{v_{n+1}} \Box \delta_x,$$
	for some $u_1, v_1, \ldots, u_{n+1}, v_{n+1}, x \in G$, some $\Lambda_i \in \D(K_i) \ (i = 1, \ldots, n+1)$, and some $\Psi_1, \ldots, \Psi_{n+1} \in \ell^1(A)^{**}$. Since $A \lhd G$, each $\Psi_i^{u_i} \in \ell^1(A)^{**}$, and each $\Gamma_i^{v_i}$ is an invariant difference on $K_i^{v_i} \leq A$ by Lemma \ref{2.2}. Any product of this form is zero by Lemma \ref{3.5a}.
\end{proof}

We now turn to the more difficult case of a just infinite group as in Condition (3) of Theorem \ref{00.4}. This case corresponds to those just infinite groups $G$ with a finite structure lattice and trivial Baer radical. We begin by proving a couple of lemmas.

\begin{lemma}		\label{3.1}
Let $M = H_1 \times \cdots \times H_n$, for some groups $H_1, \ldots, H_n$. For each $i = 1, \ldots, n$ let $A_i, B_i \leq H_i$ be either finite index or trivial, and suppose that they satisfy 
$$A_i = 1 \Leftrightarrow i > p, \qquad i>p \ \& \ B_i = 1  \Leftrightarrow i > q,$$
for some natural numbers $p \leq q$ (we may always assume this by relabeling the $H_i$). Let
$$\Lambda_1 \in \ell^1(A_1 \times \cdots \times A_n)^{**}, \qquad \Lambda_2 \in \ell^1(B_1 \times \cdots \times B_n)^{**}$$
be left-invariant. Then there exist $m \in \N$, $t_1 \ldots, t_m \in M$, and 
$$\Psi_1, \ldots, \Psi_m \in \ell^1(\Pi_{i=1}^p A_i \times \Pi_{i=p+1}^q B_i)^{**},$$
which are each left-invariant for that subgroup, and which satisfy 
$$\Lambda_1 \Box \Lambda_2 = \Psi_1 \Box \delta_{t_1} + \cdots + \Psi_m \Box \delta_{t_m},$$
when considered as elements of $\ell^1(M)^{**}$.
\end{lemma}

\begin{proof}
	Empty products in this proof can be taken to be the trivial group. Relabeling again if necessary, we may assume that there exists $r \leq p$ such that 
	$$A_i \neq 1 \ \& \ B_i \neq 1 \Leftrightarrow i \leq r.$$
	Define subgroups $P$, $Q$, and $R$ by
	$$P = \Pi_{i=1}^p A_i \times \Pi_{i=p+1}^q B_i,$$
	$$Q = \Pi_{i=1}^r A_i \cap B_i \times \Pi_{i=p+1}^q B_i,$$
	$$R = \Pi_{i=p+1}^q B_i,$$
	and note that $Q$ is a finite index subgroup of $\Pi_{i = 1}^n B_i$. 
	Let $m = [\Pi_{i=1}^n B_i : Q]$ and let $t_1, \ldots, t_m$ be a right transversal for $Q$ in $\Pi_{i=1}^n B_i$. Let
	$$\Lambda_2 = \Phi_1 \Box \delta_{t_1} + \cdots + \Phi_m \Box \delta_{t_m}$$
	be the corresponding right coset decomposition of $\Lambda_2$, and note that each $\Phi_i \ (i = 1, \ldots, m)$ is $Q$-left-invariant by Lemma \ref{0.1a}, and so also $R$-left-invariant. 
	For each $i \in \{ 1, \ldots, m \}$, both $\Phi_i$ and $\Lambda_1$ are supported on $P$, and $R \lhd P$, so by Lemma \ref{0.4b}(ii) $\Lambda_1 \Box \Phi_i$ is $R$-left-invariant as well. It is also clearly $\Pi_{i=1}^p A_i$-left-invariant, whence it is $P$-left-invariant.  The result now follows by taking $\Psi_i = \Lambda_1 \Box \Phi_i \ (i= 1, \ldots, m)$.
\end{proof}

In the next lemma we shall make use of \cite[Proposition 7]{Wi71}, which states that, given a group $G$ as in (3) of Theorem \ref{00.4}, every subnormal subgroup $K$ of $G$ is commensurable with $H_{i_1} \times \cdots \times H_{i_r},$ for some $r \in \{ 1, \ldots, n \}$ and $i_1, \ldots, i_r \in \{ 1, \ldots, n \}$ distinct; moreover the proof of that proposition shows us that $K$ is commensurable with $K\cap H_{i_1} \times \cdots \times K\cap H_{i_r}$. We shall also use the fact that $H_i \cap K$ is either trivial or finite index in $H_i$, for each $i =1, \ldots, n$; this follows from the fact that each $H_i$ is hereditarily just infinite, and hence all of its subnormal subgroups are either finite index or trivial by \cite[Proposition 4]{Wi71}.

\begin{lemma}		\label{3.2}
	Let $G$ be a just infinite group with trivial Baer radical and finite structure lattice. Let $n \in \N$ and $H_1, \ldots, H_n$ be
	as in Theorem \ref{00.4}(3), and write $M = H_1 \times \cdots \times H_n$. Let $K_1, \ldots, K_{n+1} \approxlhd G$, 
	and let $\Gamma_i \in \D(\Pi_{j=1}^n K_i\cap H_j) \ (i =1, \ldots, n+1)$. Then $\Gamma_1 \Box \cdots \Box \Gamma_{n+1} = 0$.
\end{lemma}

\begin{proof}
	We shall prove the following stronger statement by 
	induction on $q$: given $q+1$ subnormal subgroups $K_1, \ldots, K_{q+1}$ such that $\vert \{ i : H_i \cap K_j \neq 1 \text{ for some } j \} \vert \leq q$, and invariant differences $\Gamma_i$ on $\Pi_{j=1}^q K_i \cap H_j \ (i = 1, \ldots, q+1)$, we have $\Gamma_1 \Box \cdots \Box \Gamma_{q+1} = 0.$
	If $q=1$ the result follows from Lemma \ref{2.1}. 
	
	Suppose instead that $q>1$. Suppose further that  $\vert \{ i : H_i \cap K_j \neq 1 \text{ for some } j \} \vert < q$. Then by the induction hypothesis applied to $K_2, \ldots, K_{q+1}$ we have $\Gamma_2 \Box \cdots \Box \Gamma_{q+1} = 0.$ Hence we may suppose that
	$q =  \vert \{ i : H_i \cap K_j \neq 1 \text{ for some } j \} \vert$. 
	Relabelling if necessary, we may suppose that there exists a sequence of natural numbers $p_2 \leq p_3 \leq \cdots \leq p_{q+1}$ satisfying
	$$K_2 \cap H_i = 1 \Leftrightarrow i>p_2 \quad \text{ and } \quad 
	i>p_{j-1} \ \& \	K_j \cap H_i = 1 \Leftrightarrow i>p_j 
	\ (j = 3, \ldots, q+1).$$
	We claim that for each $r \in \{ 2, \ldots, q+1 \}$ we can write
	\begin{equation}		\label{eq3.1}
	\Gamma_2 \Box \cdots \Box \Gamma_r = \sum_{i=1}^{m_r} \Phi_{r, i} \Box \delta_{t_{r,i}},
	\end{equation}
	for some $m_r \in \N$, some $t_{r,1}, \ldots, t_{r, m_r} \in M$, and some elements $\Phi_{r, i} \in \ell^1(G)^{**} \ (i = 1, \ldots, m_r)$ which are supported on, and left-invariant with respect to, subgroups of the form 
	\begin{equation}	\label{eq3.3}
	\Pi_{j=1}^{p_2} (K_2^{s_2} \cap H_j) \times \cdots \times \Pi_{j=p_{r-1}+1}^{p_r} ( K_r^{s_r} \cap H_j),
	\end{equation}
	for some group elements $s_2, \ldots, s_r \in M$, which may depend on $r$ and $i$. Indeed the case $r =2$ is trivial and we proceed by induction on $r$.
	Suppose that \eqref{eq3.3} holds for $r \geq 2$. Then 
	\begin{equation}	\label{eq3.4}
	\Gamma_2 \Box \cdots \Box \Gamma_{r+1} = \left( \sum_{i=1}^{m_r} \Phi_{r, i} \Box \delta_{t_{r,i}} \right)\Box \Gamma_{r+1}
	= \sum_{i=1}^{m_r} \Phi_{r, i} \Box \Gamma_{r+1}^{t_{r, i}} \Box \delta_{t_{r,i}}.
	\end{equation}
	Note that for each $i$ and $j$ we have $K_r^{t_{r, j}} \cap H_i = 1$ if and only if $K_r \cap H_i = 1$, since conjugation by elements of $M$ preserves the direct product decomposition. Hence we may apply Lemma \ref{3.1} to each product $\Phi_{r, i} \Box \Gamma_{r+1}^{t_{r, i}}$, with 
	$$ A_l = 
	\begin{cases}
		K_k^{s_k} \cap H_l & \text{ for }  k = 1, \ldots, r, \text{ and } l = p_{k-1}+1, \ldots, p_k, \\
		1 & \text{ for }l > p_r,
	\end{cases}
	$$
	and $B_l = K_{r+1}^{t_{r,i}} \cap H_l$ for $l = 1, \ldots, n$, where $s_1, \ldots, s_r$ are as in \eqref{eq3.3}. This allows us to write $\Phi_{r, i} \Box \Gamma_{r+1}^{t_{r, i}}$ as a finite sum
	$ \sum_j \Psi_{i,j} \Box \delta_{u_{i,j}} $
	for some $u_{i,1}, u_{i,2}, \ldots \in M$, and some left-invariant elements 
	$$\Psi_{i,1}, \Psi_{i,2}, \ldots \in \ell^1(\Pi_{i=1}^{p_2} (K_2^{s_2} \cap H_i) \times \cdots 
	\times \Pi_{i=p_{r-1}+1}^{p_r} ( K_r^{s_r} \cap H_i) \times \Pi_{i=p_r+1}^{p_{r+1}} (K_{r+1}^{t_{r, i}} \cap H_i))^{**}.$$
	Putting this back into \eqref{eq3.4} gives
	$$\Gamma_2 \Box \cdots \Box \Gamma_{r+1} = \sum_{i, j} \Psi_{i,j} \Box \delta_{u_{i,j}t_{r,i}},$$
	which is a sum of the required form.
	
	We can now complete the proof. Consider \eqref{eq3.1} for $r=q+1$. We note that, for any $i \in \{1, \ldots, m_{q+1}\}$, the functional $\Phi_{q+1, i}$ is left invariant on a 
	finite index subgroup of the group $H_1 \times \cdots \times H_{q+1}$. Since $\Gamma_1$ is also supported on the latter subgroup,
	Lemma \ref{2.1} implies that $\Gamma_1 \Box \Phi_{q+1, i} = 0.$ It now follows from \eqref{eq3.1} that
	$\Gamma_1 \Box \cdots \Box \Gamma_{q+1} = 0,$ as required.
\end{proof}

We can now complete the proof of Theorem \ref{00.1} for the case of a group as in Theorem \ref{00.4}(3).

\begin{proposition}		\label{3.3}
	Let $G$ be an amenable just infinite group with trivial Baer radical and finite structure lattice. Let $n$ be as in Theorem \ref{00.4}(3). Then $\Asn(G)$ is nilpotent of index $n+1$.
\end{proposition}

\begin{proof}
	Let $H_1, \ldots, H_n$ be as in Theorem \ref{00.4}(3), and write $M = H_1 \times \cdots \times H_n$. That $\Asn(G)^{\Box n} \neq \{ 0 \}$ follows from Lemma \ref{0.4} by considering $\Gamma_1 \Box \cdots \Box \Gamma _n$, for non-zero invariant differences 
	$\Gamma_i$ on $H_i \ (i=1, \ldots, n)$. We shall show that $\Asn(G)^{\Box (n+1)} = \{ 0 \}$. 
	
	To this end let $K_1, \ldots, K_{n+1} \approxlhd G$, and let $\Gamma_i \in \D(K_i) \ (i=1, \ldots, n+1)$. For each $i$ write $m_i = [K_i : \Pi_{j=1}^n K_i \cap H_j]$, choose a right transversal $t_{i, 1}, \ldots, t_{i, m_i}$ for $\Pi_{j=1}^n K_i \cap H_j$ in $K_i$ satisfying Corollary \ref{0.1b}, and let 
	$$\Gamma_i = \sum_{j=1}^{m_i} \Phi_{i,j} \Box \delta_{t_{i,j}}$$
	be the corresponding right coset decomposition. Choosing each transversal in accordance with Corollary \ref{0.1b} ensures that each $\Phi_{i,j}$ is an invariant difference on $K_i$. We have
	\begin{equation}		\label{eq3.2}
	\Gamma_1 \Box \cdots \Box \Gamma_{n+1} = \sum_{j_1=1}^{m_1} \cdots \sum_{j_{n+1} = 1}^{m_{n+1}} \Phi_{1, j_1} \Box 
	\Phi_{2, j_2}^{t_{1, j_1}} \Box \cdots \Box \Phi_{n+1, j_{n+1}}^{t_{1, j_1} \cdots t_{n, j_n}} \Box \delta_{t_{1, j_1}}
	 \Box \cdots \Box \delta_{t_{n+1, j_{n+1}}}.
	\end{equation}
	For any indices $i$ and $j$, and any $s \in G$, Lemma \ref{2.2} implies that $\Phi_{i,j}^s$ is an invariant difference on $\Pi_{l=1}^n (K_i \cap H_l)^s = \Pi_{l=1}^n K_i^s \cap H_l$. Hence, by Lemma \ref{3.2} each of the products in the sum \eqref{eq3.2} are equal to zero, and the result follows.
\end{proof}

We are now in a position to prove Theorem \ref{00.1}.

\begin{proof}[Proof of Theorem \ref{00.1}]
	Suppose that $G$ is a branch group. Let $(H_i, L_i, k_i)_{i \in \N}$ be a branch structure for $G$, where $H_i  = L_i^{(1)} \times \cdots \times L_i^{(k_i)}$, and let $\Gamma_{i,j}$ be an invariant difference on $L_i^{(j)}$. Then by Lemma \ref{0.4} $\Gamma_{i,1} \Box \cdots \Box \Gamma_{i, k_i} \neq 0$, and, since $k_i \rightarrow \infty$ as $i \rightarrow \infty$, the algebra $\Asn(G)$ cannot be nilpotent.

	If $G$ is not a branch group, then by Theorem \ref{00.4} there are two cases to consider.
	In the case that $G$ is finitely-generated and virtually abelian, $\Asn(G)$ is nilpotent of index $\rank G +1$ by Theorem \ref{3.5b}. 
	
	Finally, suppose that $G$ is as in Theorem \ref{00.4}(3). By Proposition \ref{3.3} $\Asn(G)$ is nilpotent of index $n+1$. It follows from \cite{Wi71} that the structure lattice of $G$ is isomorphic to the lattice of subsets of $\{1, \ldots, n \}$, so that $n+1 = \log_2 \vert \mathcal{L}(G) \vert + 1$.
\end{proof}

\section{The Algebras $\Jsn(G)$ and $\rad \ell^1(G)^{**}$ When $G$ is a Branch Group}
\noindent
In this section we investigate the structure of $\Jsn(G)$ and $\rad \ell^1(G)^{**}$ when $G$ is an amenable branch group, and also prove Theorem \ref{00.2} and Theorem \ref{00.5}. We begin by fixing some notation.

Let $G$ be a group for which there exists a finite or infinite sequence of natural numbers $(k_i)$, together with either finite or infinite descending chains of subgroups $H_1 \supset H_2 \supset \cdots$ and $L_1 \supset L_2 \supset \cdots $, all of the same length, such that $k_1 \geq 2$ and such that conditions (2)-(5) of Definition \ref{00.3} hold. In this case we refer to the sequence $(H_i, L_i, k_i)$ as a \textit{partial branch structure} for $G$. Note that a branch structure is in particular a partial branch structure, and also that just infinite groups as in condition (3) of Theorem \ref{00.4} possess a partial branch structure of length 1.

Let $G$ be group which has a partial branch structure $(H_i, L_i, k_i)$. We shall write $\pi_{i,j} \colon H_i \rightarrow H_i$ for the map which deletes the direct factor $L_i^{(j)}$. Also write $P_i \colon \ell^1(G) \rightarrow \ell^1(H_i)$ for the projection map. Given $i \in \N$ and $j \in \{1, \ldots, k_i \}$ we define
$$I_{i,j} = \left\{ \Phi \in \ell^1(G)^{**} : \delta_{s} \Box P_i^{**}(\delta_t \Box \Phi) =  P_i^{**}(\delta_t \Box \Phi), \ \pi_{i,j}^{**}(P_i^{**}(\delta_{t} \Box \Phi)) = 0 \ (s \in L^{(j)}_i, t \in G) \right\}.$$
Let $t_1, \ldots, t_n$ be a left transversal for $H_i$ in $G$, let $\Phi \in \ell^1(G)^{**}$, and let $\Phi = \sum_k \delta_{t_k} \Box \Phi_k$ be its left coset decompositon.
The functional $\Phi$ belongs to $I_{i,j}$ if and only if every $\Phi_k$ is $L_i^{(j)}$-left-invariant and satisfies $\pi_{i,j}^{**}(\Phi_k) = 0$. For this reason we find it useful to also define a set $U_{i,j}$ by
$$U_{i,j} = \left\{ \Phi \in \ell^1(H_i)^{**} : \delta_s \Box \Phi = \Phi \ ( s \in L_i^{(j)} ), \ \pi_{i,j}^{**}(\Phi) = 0 \right\}.$$
By \cite[Proposition 4.2, Lemma 4.3]{W2} $U_{i,j}$ is a left ideal of $\ell^1(H_i)^{**}$ satisfying $U_{i,j}^{\Box 2} = 0$. We shall typically identify $U_{i,j}$ with its image inside $\ell^1(G)^{**}$. 
We shall show in due course that $I_{i,j}$ is always a nilpotent left ideal of $\ell^1(G)^{**}$, and we shall use ideals of this form to prove Theorem \ref{00.2}. The next lemma establishes that $I_{i,j}$ is indeed a left ideal.

\begin{lemma}		\label{5.1}
	Let $i \in \N$ and $j \in \{1, \ldots, k_i \}$. Then
	\begin{enumerate}
		\item[\rm (i)] the set $I_{i,j}$ is the left ideal of $\ell^1(G)^{**}$ generated by $U_{i,j}$ and is weak*-closed;
		\item[\rm (ii)] we have $\D(L_i^{(j)}) \subset U_{i,j}$.
	\end{enumerate}
\end{lemma}

\begin{proof}
	(i) The fact that $I_{i,j}$ is weak*-closed follows from the fact that $P_i^{**}$, $\pi_{i,j}^{**}$, and multiplication by $\delta_t$ for some $t \in G$ are all weak*-continuous. Given $\Lambda \in \ell^1(G)^{**}$ and $ \Psi \in U_{i,j}$, let $\Lambda = \sum_{k=1}^n \delta_{t_k} \Box \Lambda_k$ be a left coset decomposition for $\Lambda$ with respect to $H_i$. Then $\Lambda \Box \Psi = \sum_{k=1}^n \delta_{t_k} \Box \Lambda_k \Box \Psi$, and each $\Lambda_k \Box \Psi \in U_{i,j}$ since $U_{i,j}$ is a left ideal in $\ell^1(H_i)^{**}$. Hence $\ell^1(G)^{**} \Box U_{i,j} \subset I_{i,j}$. The reverse inclusion is immediate.
	
	(ii) This follows from the observation that for $\Phi \in \ell^1(L_i^{(j)})^{**}$ we have $\pi_{i,j}^{**}(\Phi) = \langle \Phi, 1 \rangle \delta_e$.
\end{proof}

\begin{lemma}		\label{0.2}
	Let $G$ be a group which has a partial branch structure  $(H_i, L_i, k_i)$. Fix $k_p$ belonging to $(k_i)$, let $t \in G$, and suppose that $\sigma \in S_{k_p}$ satisfies
	$$tL_p^{(j)}t^{-1} = L_p^{(\sigma(j))} \quad (j=1, \ldots, k_p).$$
	\begin{enumerate}
		\item[\rm (i)] We have
		$$U_{p,j}^t = U_{p, \sigma(j)} \quad (j=1, \ldots, k_p).$$
		\item[\rm (ii)] Suppose that $\Phi \in \ell^1(H_p)^{**}$ is supported on $L_p^{(j)}$. Then $\Phi^t$ is supported on $L_p^{(\sigma(j))}$.
	\end{enumerate}
\end{lemma}

\begin{proof}
	(i) Fix $j$, and write $k = k_p$ and $\pi_i = \pi_{p,i} \ (i = 1, \ldots, k)$ in order to simplify notation. Let $\alpha \in \operatorname{Aut}(H_p)$ denote conjugation by $t$. We first show that $\alpha \pi_j \alpha^{-1} = \pi_{\sigma(j)}$. We know that $\alpha$ may be written as $\beta \circ \gamma$, where $\beta = (\beta_1, \ldots, \beta_k)$ for some $\beta_i \in \operatorname{Aut}(L_p^{(i)}) \ (i = 1, \ldots, k)$, and $\gamma$ acts by permuting the factors $L_p^{(i)}$ via $\sigma$, that is 
	$$\gamma(u_1, \ldots, u_k) = (u_{\sigma^{-1}(1)}, \ldots, u_{\sigma^{-1}(k)}) \qquad (u_i \in L^{(i)}_p, \ i=1, \ldots, k).$$
	We now calculate that, for $u_i \in L^{(i)}_p \ (i = 1, \ldots, k)$, we have
	\begin{align*}
	\alpha \pi_j \alpha^{-1} (u_1, \ldots, u_k) &= \alpha \pi_j (\beta^{-1}_{\sigma(1)}(u_1), \ldots, \beta^{-1}_{\sigma(k)}(u_k)) \\
	&= \alpha (\beta^{-1}_{\sigma(1)}(u_1), \ldots, e, \ldots, \beta^{-1}_{\sigma(k)}(u_k))
	&\text{where }e\text{ appears in the }j^{\rm th}\text{ place} \\
	&= (u_1, \ldots, e, \ldots, u_k) 
	&\text{where }e\text{ appears in the }\sigma(j)^{\rm th}\text{ place}\\
	&= \pi_{\sigma (j)}(u_1, \ldots, u_k),
	\end{align*}
	as required. 
	
	Let $\Phi  \in U_{p,j}$, and let $s \in L^{(\sigma(j))}_p$. Then 
	$$
	\delta_s \Box \alpha^{**}(\Phi) = \alpha^{**}(\delta_{\alpha^{-1}(s)} \Box \Phi) = \alpha^{**}(\Phi). 
	$$
	Also $\pi_j^{**}\alpha^{**}(\Phi) = \alpha^{**}\pi_{\sigma(j)}^{**}(\Phi) = \alpha^{**}(0) = 0.$ Hence $\alpha^{**}(U_{p,j}) \subset U_{p,\sigma(j)}$. 
	By symmetry, we also have $(\alpha^{-1})^{**}(U_{p,\sigma(j)}) \subset U_{p,j}$, which implies that 
	$\alpha^{**}(U_{p,j}) = U_{p, \sigma(j)}$, as required.
	
	(ii) Argue as in Lemma \ref{2.2}.
\end{proof}

\begin{lemma}		\label{0.3}
	Let $G$ be a group which has a partial branch structure $(H_i, L_i, k_i)$. Fix $k = k_p$ belonging to $(k_i)$. Suppose that $\Psi_1, \ldots, \Psi_k \in \ell^1(G)^{**}$ and that each $\Psi_j$ belongs to one of the sets $U_{p,i} \ (i =1, \ldots, k)$. 
	Let $s_1, \ldots, s_{k+1} \in G$. Then
	\begin{enumerate}
		\item[\rm (i)] $\Psi_1^{s_1} \Box \cdots \Box \Psi_{k+1}^{s_{k+1}} = 0;$
		\item[\rm (ii)] $(\Psi_1 \Box \delta_{s_1}) \Box \cdots \Box (\Psi_{k+1} \Box \delta_{s_{k+1}}) = 0.$
	\end{enumerate}
\end{lemma}

\begin{proof}
	(i) By Lemma \ref{0.2}(i) the sets $U_{p,i}^{s_j}$ are also of the form $U_{p,l}$ for some $l \in \{1, \ldots, k \}$, and so by a pigeon hole argument two of the elements 
	of $\{\Psi_1^{s_1}, \ldots, \Psi_{k+1}^{s_{k+1}}\}$ come from the same set $U_{p,l}$. Since each $U_{p,l}$ is an ideal in $\ell^1(H_p)^{**}$, it follows that the product $\Psi_1^{s_1} \Box \cdots \Box \Psi_{k+1}^{s_{k+1}}$ is in fact the  product of two elements of some $U_{p,l}$. Since $U_{p,l}^{\Box 2} = \{ 0 \}$, the above product must also be zero.
	
	(ii) We have 
	$$(\Psi_1 \Box \delta_{s_1}) \Box \cdots \Box (\Psi_{k+1} \Box \delta_{s_{k+1}}) = \Psi_1 \Box \Psi_2^{s_1} \Box \cdots \Box \Psi_{k+1}^{s_1s_2 \cdots s_k} \Box \delta_{s_1s_2 \cdots s_{k+1}},$$
	so the result follows from part (i).
\end{proof}

The main result of this section is the next proposition, from which Theorem \ref{00.2} will follow by a short argument. It is also an important ingredient in the proof of Theorem \ref{00.5}.

\begin{proposition}		\label{5.3}
	Let $G$ be an amenable group with a partial branch structure $(H_i, L_i, k_i)$. For each $p \in \N$ at most the length of $(k_i)$ and each $j \in \{1, \ldots, k_p \}$ the left ideal $I_{p,r}$ is nilpotent of index $k_p+1$. The same holds for $\Jsn(G) \cap I_{p,r}$.
\end{proposition}

\begin{proof}
	To ease notation fix $p \in \N$, and write $k= k_p$ and $n= [G: H_p]$. 	We shall first show that $I_{p,r}^{\Box (k+1)} = \{ 0 \}$. Let $t_1, t_2, \ldots, t_n$ be a right transversal for $H_p$ in $G$. Let $\mu_1, \ldots, \mu_{k+1} \in U_{p,r}$, and let $\Phi_1, \ldots, \Phi_{k+1} \in \ell^1(G)^{**}$. We must show that 
	$$\Phi_1 \Box \mu_1 \Box \cdots \Box \Phi_{k+1} \Box \mu_{k+1} = 0.$$
	
	Let the right coset decomposition of $\Phi_i$ be
	$$\Phi_i = \Phi_{i, 1} \Box \delta_{t_1} + \cdots + \Phi_{i, n} \Box \delta_{t_n},$$
	Then for each $i = 1, \ldots, k+1$, we have
	\begin{align}		\label{eq2}
	\Phi_i \Box \mu_i &= \Phi_{i, 1} \Box \delta_{t_1} \Box \mu_i + \cdots + \Phi_{i, n} \Box \delta_{t_n} \Box \mu_i \\
	&= \Phi_{i, 1} \Box \mu_i^{t_1} \Box \delta_{t_1} + \cdots + \Phi_{i, n} \Box \mu_i^{t_n} \Box \delta_{t_n}. \notag
	\end{align}
	By Lemma \ref{0.2}(i) each element $\mu_i^{t_j}$ belongs to $U_{p,q}$ for some $q \in \{ 1, \ldots, k \}$. Since
	these sets are ideals in $\ell^1(H_p)^{**}$ it follows from \eqref{eq2} that each $\Phi_i \Box \mu_i$ has the form
	\begin{equation}		\label{eq1}
	\Phi_i \Box \mu_i = \Psi_{i, 1} \Box \delta_{t_1} + \cdots + \Psi_{i, n} \Box \delta_{t_n}
	\end{equation}
	where each $\Psi_{i,j}$ belongs to a set of the form $U_{p,q}$ for some $q \in \{1, \dots, k\}$. The product 
	$\Phi_1 \Box \mu_1 \Box \cdots \Box \Phi_{k+1} \Box \mu_{k+1}$ can then be multiplied out using \eqref{eq1}, and the result is a sum of products of the form considered in Lemma \ref{0.3} (ii), each of which is zero by that lemma. Hence $\Phi_1 \Box \mu_1 \Box \cdots \Box \Phi_{k+1} \Box \mu_{k+1} = 0$, as required.
	
	To complete the proof it suffices to show that $(\Jsn(G) \cap I_{p,r})^{\Box k} \neq \{ 0 \}$. Let $\Gamma \in \D(L_p^{(r)}) \setminus \{ 0 \}$, which belongs to $U_{p,q}$ by Lemma \ref{5.1}(ii). Take $x_1, \ldots, x_k \in G$ such that $x_i L^{(r)}_p x_i^{-1} = L^{(r+i)}_p$ for $i=1, \ldots, k$ (where the superscript is understood modulo $k$). Define a sequence of group elements 
	$s_1, \ldots, s_k \in G$ by $s_1 = x_1$  and $s_i = x_{i-1}^{-1}x_i$ for $i=2, \ldots, k$, 
	so that, for each $i \in  \{1, \ldots, k\},$ we have $s_1 \cdots s_i = x_i$. For all $i$ we have $\delta_{s_i} \Box \Gamma \in I_{p,r} \cap \Jsn(G)$ and 
	$$(\delta_{s_1}\Box \Gamma) \Box \cdots \Box (\delta_{s_k} \Box \Gamma)
	=(\Gamma^{s_1} \Box \Gamma^{s_1 s_2} \Box \cdots \Box \Gamma^{s_1 \cdots s_k}) \Box \delta_{s_1 \cdots s_k}
	= (\Gamma^{x_1} \Box \Gamma^{x_2} \Box \cdots \Box \Gamma^{x_k}) \Box \delta_{x_k},$$
	which is non-zero since  $\Gamma^{x_1} \Box \Gamma^{x_2} \Box \cdots \Box \Gamma^{x_k} \neq 0$ by Lemma \ref{0.2}(ii) and Lemma \ref{0.4}, and $\delta_{x_k}$ is invertible.
\end{proof}

We can now prove Theorem \ref{00.2}.

\begin{proof}[Proof of Theorem \ref{00.2}.]
	Let $(H_i, L_i, k_i)_{i \in \N}$ be a branch structure for $G$, and let $p \in \N$. Since $k_p+1$ tends to infinity as $p \rightarrow \infty$, the previous proposition immediately implies that $\ell^1(G)^{**}$ and $\Jsn(G)$ contain nilpotent left ideals of arbitrarily large index. Recall that nilpotent left ideals of an algebra are contained in its Jacbson radical. A theorem of Grabiner \cite{G69} then implies that if every element of $\rad \ell^1(G)^{**}$ were nilpotent then $\rad \ell^1(G)^{**}$ itself would be nilpotent, and hence we conclude that $\rad \ell^1(G)^{**}$ must contain non-nilpotent elements.
\end{proof}

Theorem \ref{00.2} still leaves open the possibility that every element of $\rad \ell^1(G)^{**}$ is either nilpotent of some bounded index, or else is not nilpotent. However this cannot happen by the next proposition. 

\begin{proposition}		\label{5.2}
	Let $G$ be an amenable branch group. Then $\Asn(G)$, $\Jsn(G)$, and $\rad \ell^1(G)^{**}$ contain nilpotent elements of arbitrarily large index.
\end{proposition}

\begin{proof}
	Let $(H_i, L_i, k_i)_{i \in \N}$ be a branch structure for $G$, and fix $p \in \N$. Let $\Gamma_i \in {\D(L_p^{(i)}) \setminus \{ 0 \}}$ for ${i = 1, \ldots, k_p}$, and let $\Phi = \Gamma_1 + \cdots + \Gamma_{k_p}$. Clearly $\Phi \in \Asn(G) \subset \Jsn(G)$, and it belongs to $\rad \ell^1(G)^{**}$ since each $\Gamma_i \in I_{p,i} \ (i = 1, \ldots, k_p)$, which is a radical left ideal by Proposition \ref{5.3}. We shall show that $\Phi$ is nilpotent of index $k_p + 1$. Indeed 
	$$\Phi^{\Box (k_p+1)} = \sum_{i_1} \cdots \sum_{i_{k_p+1}} \Gamma_{i_1} \Box \cdots \Box \Gamma_{i_{k_p+1}},$$
	which is zero by Lemma \ref{0.3}(i), since each $\Gamma_i \in U_{p,i}$. Similarly, any product of the form $\Gamma_{i_1} \Box \cdots \Box \Gamma_{i_{k_p}}$ will be zero, except when $i_1, \ldots, i_{k_p}$ are all distinct and hence equal to $1, \ldots, k_p$ in some order. The $\Gamma_i$ commute with each other since their supports do, and hence 
	$$\Phi^{\Box k_p} = k_p! \Gamma_1 \Box \cdots \Box \Gamma_{k_p},$$
	which is non-zero by Lemma \ref{0.4}. Since $k_p \rightarrow \infty$ as $p \rightarrow \infty$ this completes the proof. 
\end{proof}

Finally, we prove Theorem \ref{00.5}.

\begin{proof}[Proof of Theorem \ref{00.5}]
	Let $G$ be a just infinite group for which 
	$(\rad \ell^1(G)^{**})^{\Box 2} = \{ 0 \}$. We shall consider each of the three possibilities arising in Theorem \ref{00.4}. The group $G$ cannot be a branch group by Theorem \ref{00.2}. Suppose that $G$ is virtually abelian. Then by Theorem \ref{3.5b} $\Jsn(G)$ is a radical left ideal, and is nilpotent of index $\rank G +1$, whence $\rank G = 1$. Therefore $G$ has a finite index copy of $\Z$, which is hereditarily just infinite, and hence $G$ is hereditarily just infinite by Lemma \ref{5.4}.
	
	Finally, suppose that $G$ satisfies condition (3) of Theorem \ref{00.4}, and let $H= H_1 \times \cdots \times H_n$ be as written there. If $n \geq 2$ then $(H, H_1, n)$ is a partial branch structure of length 1, so that, by Proposition \ref{5.3}, $\ell^1(G)^{**}$ contains a radical left ideal which is nilpotent of index $n+1 \geq 3$, contradicting our assumption on $\rad \ell^1(G)^{**}$. Therefore $n=1$, and $G$ contains a finite index hereditarily just infinite subgroup, and hence is hereditarily just infinite itself by Lemma \ref{5.4}.  
\end{proof}

\subsection*{Acknowledgements}
\noindent
Part of this work was undertaken during my postdoctoral position at the Laboratoire des Math{\' e}matiques de Besan{\c c}on, and I would like to thank Uwe Franz and Yulia Kuznetsova for their kind hospitality during that period. I am extremely grateful to Bence Horv{\' a}th for all of his encouragement and insight, and for taking time out of his holiday to read a draft of this manuscript. I would also like to thank Bruno Duchesne for a helpful email exchange. Finally, I would like to thank the referee for their very helpful comments.


\begin{thebibliography}{99}
	
	\bibitem{BGS} L.\ Bartholdi, R.\ Grigorchuk, and Z.\ Sunik, Branch groups, \emph{Handbook of algebra}, Vol. 3, North-Holland, Amsterdam (2003), 989--1112.
	
	\bibitem{BKN} L.\ Bartholdi, V.\ Kaimanovich, and V.\ Nekrashevych, On amenability of automata groups, \emph{Duke Math. J.} \textbf{154} (2010), 575--598.
	
	\bibitem{D} H.\ G.\ Dales, \emph{Banach algebras and automatic continuity}, London Math. Soc. Monographs, Volume 24, Clarendon Press, Oxford, 2000.
	
	\bibitem{DL} H.\ G.\ Dales and A.\ T-M.\ Lau, \emph{The second duals of Beurling algebras}, Mem. Amer. Math. Soc., Volume 177, 2005.
	
	\bibitem{DTW} B.\ Duchesne, R.\ Tucker-Drob, and P.\ Wesolek, CAT(0) cube complexes and inner amenability, \textit{Groups Geom. Dyn.} \textbf{15} (2021), 371--411.
	
	\bibitem{G69} S.\ Grabiner, The nilpotency of Banach nil algebras, \emph{Proc. Amer. Math. Soc.} \textbf{21} (1969), 510.
	
	\bibitem{Gr00} R.\ I.\ Grigorchuk, Just infinite branch groups, \emph{New horizons in pro-p groups} (Markus P. F. du Sautoy
	Dan Segal and Aner Shalev, eds.), Birkh{\"a}user Boston, Boston, MA, 2000, pp. 121--179.
	
	\bibitem{JM} K.\ Juschenko and N.\ Monod, Cantor systems, piecewise-translations and amenable simple groups, \emph{Annals of Math.} \textbf{178} (2013), 775--787.
	
	\bibitem{Pa94} T.\ W.\ Palmer, \emph{Banach algebras and the general theory of *-algebras, Vol.I}, volume 49 of \emph{Encyclopedia of Mathematics and its applications}. Cambridge University Press, Cambridge, 1994. Algebras and Banach Algebras.
	
	\bibitem{W2} J.\ T.\ White, The radical of the bidual of a Beurling algebra, \emph{Q. J. Math.} \textbf{69} (2018), 975--993.
	
	\bibitem{Wi71} J.\ S.\ Wilson, Groups with every proper quotient finite, \emph{Proc. of the Cambridge Philosophical Soc.} \textbf{69} (1971), 373--391.
	
\end{thebibliography}
\end{document}